\documentclass{amsart}
\usepackage{amssymb}
\usepackage{amsmath, amsfonts}
\usepackage{setspace}
\usepackage[super]{nth}
\usepackage{xcolor}
\usepackage{mathtools}
\usepackage{relsize}
\usepackage{booktabs}
\usepackage[colorinlistoftodos]{todonotes}

\newtheorem{theorem}{Theorem}[section]

\newtheorem{lemma}[theorem]{Lemma}

\newtheorem{cor}[theorem]{Corollary}
\theoremstyle{definition}

\newtheorem{remark}[theorem]{Remark}

\numberwithin{equation}{section}

\begin{document}
\title[Distribution of differences of characters  ]
{Distribution of differences of characters evaluated at consecutive polynomial values   }

\author{Nilanjan Bag}
 \address{Department of Mathematics, Thapar Institute of Engineering and Technology, Patiala, Punjab-147004, India}
\email{nilanjanb2011@gmail.com}
\author{Dwaipayan Mazumder}
 \address{Chennai Mathematical Institute, Kelambakkam, Siruseri, Tamil Nadu 603103, India
}
\email{dwmaz.1993@gmail.com}
\date{26th January, 2026}
\begin{abstract}

In this paper, we study the distribution of difference of multiplicative and additive characters modulo $p$ at consecutive polynomial values. More precisely, for an interval $I$ over finite field and $0<m<1$, we investigate the following sums
\begin{align*}
    \sum_{n\in I}|\psi(F(n))-\psi(F(n+1))|^{2m} \quad \text{and} \quad \sum_{n\in I}|\chi(F(n))-\chi(F(n+1))|^{2m},
\end{align*}
where $\psi$ is a non-trivial additive character and $\chi$ is a non-trivial multiplicative character modulo $p$, under suitable conditions on $\chi$ and $F$. As a consequence, we derive a formula for the first moment by specializing to $m=1/2$.
\end{abstract}
\maketitle
\section{Introduction}

Let $f: \mathbb{N} \longrightarrow \mathbb{C}$ be an arithmetic function. It is interesting to study the following sum
\begin{align}\label{theproblem}
     \sum_{n\sim N}|f(n)-f(n+1)|. 
\end{align}
More specifically, we can ask how \eqref{theproblem} grows with $N$. For monotone function $f$ the problem becomes trivial but the interesting case is when $f$ oscillates. Even for classical arithmetic functions this problem appears to be highly non-trivial. For example, if we consider $f(n) = \mu(n)$, the M\"{o}bius function, then \eqref{theproblem} becomes 
$
\sum_{n\sim N}|\mu(n)-\mu(n+1)|.
$
The crude bound is $\ll N$ but no meaningful non-trivial bound is currently available. More specifically, we can not say anything about the growth as $N$ varies. Not only for arithmetic functions, even if we assume little more nicer condition on $f$, for example differentiability of $f$, it is not really straightforward to say how this sum \eqref{theproblem} varies with $N$. The presence of absolute value makes cancellation difficult to exploit. We may attempt a standard approach to provide an upper bound using Cauchy-Schwarz inequality.
$$
\sum_{n\sim N}|f(n)-f(n+1)| \leq N^{1/2}{(\sum_{n\sim N}|f(n)-f(n+1)|^2)}^{1/2}.
$$
This reduces the sum to second moment 
$$
\sum_{n\sim N}|f(n)-f(n+1)|^2,
$$
which is more traceable using $|z|^2 = z\overline{z}$ which allows for algebraic expansion and probable cancellation. Moreover if we can show that 
$
\sum\limits_{n\sim N}|f(n)-f(n+1)|^2 \ll N^{1-\epsilon},
$
for some $\epsilon > 0$ we will have some savings.
\par Despite this natural motivation, the sum of difference of consecutive primes has received significant attention. Let $p_n$ be the $n$-th prime number and $d_n$ denote the difference $p_{n+1}-p_{n}$. Heath Brown \cite{HB} unconditionally proved that, for any fixed $\epsilon >0$ we have 
$$
\sum\limits_{\substack{p_n\leq x \\d_n\geq \sqrt{p_n}}}d_n \ll_{\epsilon} x^{3/5 + \epsilon}.
$$
We also have the results for nontrivial upper bound of the sum $\sum_{p_n\leq x}d_n^2$. For more details, one can see the subsequent works of  Selberg \cite{selberg}, H. Brown \cite{HB1, HB2}, Yu \cite{yu}, Maynard \cite{JM} {\it{et.al.}}

\par The study of moments of differences of arithmetic functions has received relatively little attention. One notable exception is the work of Zhang  \cite{zhang2}, who investigated 
\begin{align*}
    M(q,k,c)=\sum_{\substack{a=1\\(a,q)=1}}^{q}\sum_{\substack{b=1\\(b,q)=1\\ab\equiv c\bmod q}}^{q}(a-b)^{2k},
\end{align*}
where $q>2$ and $c$ are two integers with $(c,q)=1$. He used estimates of trigonometric and Kloosterman sums to obtain an asymptotic formula for $M(q,k,c)$. For more study on $M(q,k,c)$, one can see the work of Zhang and Liu \cite{zhang1}.

\par
In this paper we study different but related direction by studying differences of additive and multiplicative characters evaluated at consecutive polynomial values over finite fields. Let $p$ be an odd prime and $I\subseteq \mathbb{F}_p$, finite field with $p$ elements, be an interval. We also let $|I|$ be the cardinality of the interval $I$. With this setup we investigate the following sums 
$$
 \sum_{n\in I}|\psi((F(n))-\psi(F(n+1))| \quad \text{and} \quad \sum_{n\in I}|\chi((F(n))-\chi(F(n+1))|,
$$
where $\psi$, $\chi$, $F(n)$ are respectively an additive character, a multiplicative character on $\mathbb{F}_p$ and a polynomial with integer coefficients with certain conditions. First we define,
$$
c_0 \coloneqq 1 \quad \text{and} \quad c_k \coloneqq \binom{m}{k} \coloneqq \frac{m(m-1)(m-2)\cdot\cdot\cdot(m-k+1)}{k!} \quad \text{for any} \quad k\geq 1.
$$
Then we prove the following theorems.
\begin{theorem}\label{main-1}
Let $\chi$ be a multiplicative character on $\mathbb{F}_p^{\times}$ of order $t>2$ and $I\subseteq \mathbb{F}_p^{\times}$ be an interval. Let $F$ be a monic polynomial in $\mathbb{Z}[X]$ with degree $d>1$ such that $\frac{F(X+1)}{F(X)}$ modulo $p$ is not proportional to a $t$-th power in $\mathbb{F}_p(X)$ and the real part $|\Re{(\chi(F(n+1))\overline{\chi}(F(n))}|<1,$ $\forall ~ n\in I$. Fix $0<m<1$. Then we have the following formula 
\begin{align*}
\sum_{n\in I}|\chi(F(n))-\chi(F(n+1))|^{2m}= C(\chi,m) \cdot |I|+ \mathcal{O}_{d,m}(\sqrt{p}\log p),
\end{align*}
the error term is non trivial, provided
$\sqrt{p}\log p< |I| < p.$ \par Here $C(\chi,m)=2^m\left(1-\frac{u_1}{2}+\frac{u_2}{2^2}-\frac{u_3}{2^3}+\cdot\cdot\cdot\right)$,  $u_k \coloneqq c_k\sum\limits_{\substack{r=0\\ k\equiv 2r \pmod t}}^{k}\binom{k}{r}.$

\end{theorem}
As a direct corollary of Theorem \ref{main-1}, putting $m=1/2$, we have
\begin{cor}
    Assuming all conditions in Theorem \ref{main-1}, we have 
    \begin{align*}
\sum_{n\in I}|\chi(F(n))-\chi(F(n+1))|= C(\chi,1/2) \cdot |I|+ \mathcal{O}_{d}(\sqrt{p}\log p),
\end{align*}
the error term is non trivial, provided
$\sqrt{p}\log p< |I| < p$ and $C(\chi,1/2)$ is defined as before.
\end{cor}

Next we take any arbitrary non-trivial additive character $\psi$ over the finite field $\mathbb{F}_p$. 
\begin{theorem}\label{additivechar}
    Let $\psi$ be a non-trivial additive character on $\mathbb{F}_p$, {\it i.e.} $\psi(x)=e_p{(ax)}=exp({\frac{2\pi i ax}{p}})$ for some $a\in \mathbb{F}_p^{\times}$ and $I\subseteq \mathbb{F}_p$ be an interval. Let $F(X)\in \mathbb{Z}[X]$ be a polynomial of degree $d>1$ such that $F(n)\neq F(n+1) \bmod p$, $\forall ~n\in I.$ Then we have
\begin{align*}
\sum_{n \in I}|\psi(F(n))-\psi(F(n+1))|^{2m} = D(\psi, m) \cdot |I| + \mathcal{O}_{d,m}\big(\sqrt{p}\log p\big),
\end{align*}
which is non-trivial provided
    $\sqrt{p}\log p<|I|<p$. \par Here $D(\psi,m)=2^m\left(1-\frac{v_1}{2}+\frac{v_2}{2^2}-\frac{v_3}{2^3}+\cdot\cdot\cdot\right)$,  $v_k \coloneqq c_k\sum\limits_{\substack{r=0\\ k\equiv 2r \pmod p}}^{k}\binom{k}{r}.$

\end{theorem}
As before, putting $m=1/2$, we have a direct corollary of Theorem \ref{additivechar},
\begin{cor}
    Assuming all conditions in Theorem \ref{additivechar}, we have 
    \begin{align*}
\sum_{n\in I}|\psi(F(n))-\psi(F(n+1))|= D(\psi,1/2) \cdot |I|+ \mathcal{O}_{d}(\sqrt{p}\log p),
\end{align*}
the error term is non trivial, provided
$\sqrt{p}\log p< |I| < p$ and $D(\chi,1/2)$ is as defined in Theorem \ref{additivechar}.
\end{cor}
\begin{remark}
   Here $C(\chi,m)$ and $D(\psi,m)$ both are of size less than $2^{2m}$. If we consider $C(\chi, m)$ we have 
    \begin{align*}
      C(\chi, m)&=  2^m\left(1-\frac{u_1}{2}+\frac{u_2}{2^2}-\frac{u_3}{2^3}+\cdot\cdot\cdot\right)\\
      &=2^m\left(1-my_1+m\frac{m(m-1)}{2!}y_2-\frac{m(m-1)(m-2)}{3!}y_3+\cdots\right),
    \end{align*}
    where $y_k=\frac{1}{2^k}\sum\limits_{\substack{r=0\\ k\equiv 2r \bmod t}}^{k}\binom{k}{r}.$ As $0<m<1$, we have $j-m\geq 0$ for $j\geq 1$. This gives 
 $  
     -\left(my_1+\frac{m(1-m)}{2!}y_2+\frac{m(1-m)(2-m)}{3!}y_3+\cdots\right)\leq 0<2^m-1,  
   $ and consequently, 
   \begin{align*}
      2^m\left(1-my_1+m\frac{m(m-1)}{2!}y_2-\frac{m(m-1)(m-2)}{3!}y_3+\cdots\right)<2^{2m}.
   \end{align*}
   Similarly, one can show  $D(\psi, m)<2^{2m}$.
   
\end{remark}

\section{Sketch of the proof}\label{sketch}
We first consider the following sum 
\begin{align*}
\sum_{n\in I}|f(n)-f(n+1)|^{2m},
\end{align*}
where $0< m< 1$ and $f$ is an arithmetic function satisfying $|f(n)| = 1$ $\forall ~n$. Expanding the square, we have 
\begin{align*}
    \sum_{n\in I}|f(n)-f(n+1)|^{2m}
    =2^m\sum_{n\in I}\left(1 - \frac{f(n+1)\overline{f(n)} + \overline{f(n+1)}f(n))}{2}\right)^{m}.
\end{align*}
We observe that the summand has a power series expansion if the real part
\begin{align}\label{condition}
    |\Re{(f(n+1)\overline{f(n)})}| < 1.
\end{align}
 Let us denote $\frac{f(n+1)\overline{f(n)} + \overline{f(n+1)}f(n))}{2}$ as $A(f,n)$. If $|A(f,n)| < 1$ then we have 
\begin{align}\label{seriesexp}
    &2^m\sum_{n\in I}{\left(1 - A(f, n)\right)}^{m}\\
    =&2^m\sum_{n\in I}{\left(1 - mA(f, n) + \frac{m(m-1)}{2!}A(f, n)^2 - \frac{m(m-1)(m-2)}{3!}A(f, n)^3 + \cdot \cdot \cdot\right)}.\notag
\end{align}
 We know the following asymptotic
$$
\binom{m}{k} \coloneqq \frac{m(m-1)(m-2)\cdot\cdot\cdot(m-k+1)}{k!} = \frac{(-1)^k}{\Gamma(-m)k^{1+m}}(1 + o(1)),
$$
as $k \longrightarrow \infty,$ and $\Gamma $ is extended gamma function on complex numbers.

Evidently, next we evaluate the sum 
\begin{align}\label{sum-today}
    \sum_{n\in I}A(f,n)^{k}
\end{align}
for $k \geq 1$. As our series in bracket in \eqref{seriesexp} is absolutely convergent we are free to rearrange the terms. For our purpose we take  $f(n)=\chi(F(n))$ and $f(n)=\psi(F(n))=\exp{(2\pi i\cdot \frac{aF(n)}{p}})$ respectively, where $\gcd (a, p) = 1$. To have an expansion like \eqref{seriesexp} we now specialize to the cases of interest.
\subsection{Case 1:} Let $f(n)=\psi(F(n))$. To satisfy \eqref{condition} we require the condition $|\Re{(\psi(F(n+1) - F(n)))}| <1$. By the definition of $\psi$ it in turns demands that 
\begin{align*}
\left|\Re{\left(\exp{\left(2 \pi i\frac{a(F(n+1)-F(n))}{p}\right)}\right)}\right|
= \left|\cos{\left(2 \pi \frac{a(F(n+1)-F(n))}{p}\right)}\right|<1.
\end{align*}
So we need to exclude the following cases 
$$
\cos{\left(2 \pi \frac{a(F(n+1)-F(n))}{p}\right)} = 1 \quad \text{and} \quad \cos{\left(2 \pi \frac{a(F(n+1)-F(n))}{p}\right)} = -1.
$$
The first case implies
$$
F(n+1) - F(n) = \frac{lp}{a},
$$
where $l$ is an integer. Here we impose the condition $F(n+1) - F(n) \not\equiv 0 \pmod{p} $ to discard the case. For the second case, we need to look at 
$$
F(n+1) - F(n) = \frac{(2l+1)p}{2a},
$$
where $l$ is an integer. This never happens as the left hand side is an integer. Next we are interested in estimating
\begin{align*}
\sum_{n\in I}A(\psi(F),n)^{k}=\frac{1}{2^{k}}\sum_{n\in I}\left(\psi(F(n))\overline{\psi}(F(n+1))+\psi(F(n+1))\overline{\psi}(F(n))\right)^k.
\end{align*}
\subsection{Case: 2} Let $f(n) = \chi(F(n))$. For this case, to have the condition \eqref{condition} we require $|\Re{(\chi(F(n+1))\overline{\chi}(F(n))}| <1$, which is assumed throughout in Theorem \ref{main-1}. For this setting we are required to estimate,
\begin{align*}
\sum_{n\in I}A(\chi(F),n)^{k}=\frac{1}{2^{k}}\sum_{n\in I}\left(\chi(F(n))\overline{\chi}(F(n+1))+\chi(F(n+1))\overline{\chi}(F(n))\right)^k.
\end{align*}
\par The calculations for both cases are approximately the same. However, dealing with the multiplicative character $\chi$ is a little different. In the process of proving our theorems, we will have the following type of sums
$$
\sum\limits_{n\in I}\psi(P(n)), \quad \text{and} \quad \sum\limits_{n\in I}\chi(P(n)),
$$
where $P(n)$ is a polynomial or rational function with some restriction.

For the first sum we use the completion technique. For the second one, however we can not use the completion technique directly. For multiplicative case we will get a mixed character sum of the form 
\begin{align*}
    \sum_{n}\chi(P(n))e_p(-bn), ~b\in\mathbb{F}_p,
\end{align*}
where we cannot use Weil bound directly. Hence we use recent result of Fouvry, Kowalski, Michel, Raju, Rivet and Soundararajan given in Theorem \ref{sliding-sum}. The detailed discussion is given in Section \ref{sec-5}.

\section{notation and preliminaries}
Throughout the paper, we consider the following notations:
\begin{enumerate}
\item We have $f=\mathcal{O}(g)$ or $f\ll g$ for $f\leq c\cdot g$, where $c$ is a constant which may sometimes depends upon some $\epsilon$.
\item $\binom{n}{k}$ is the $k$-th binomial coefficient of $(1+x)^n$.
\item Throughout we have used $(a,b)$ for $\gcd(a,b)$.
\item For any function $\varphi:\mathbb{Z}\rightarrow\mathbb{C}$ which is $m$-periodic, we denote the normalized Fourier transform of $\varphi$ by $\Hat{\varphi}:\mathbb{Z}\rightarrow\mathbb{C}$, defined as 
\begin{align*}
    \Hat{\varphi}(h)=\frac{1}{\sqrt{m}}\sum_{x\bmod m }\varphi(x)e\left(\frac{hx}{m}\right).
\end{align*}
\end{enumerate}  
%
Let $\varphi(n)$ be a complex valued function on $\mathbb{Z}/{m\mathbb{Z}}$, which can also be seen as function on $\mathbb{Z}$ by composing with the reducing modulo $m$. The problem of estimating the sum non-trivially over some interval is important in many aspects in analytic number theory. We are interested in the average value $\sum_{I}\varphi(n)$ over some interval $I\subset \mathbb{Z}$ of length $N$. The recent result of Fouvry, Kowalski, Michel, Raju, Rivet and Soundararajan plays vital role in estimating the error term in Theorem \ref{main-1}.
\begin{theorem}\cite[Theorem 1.1]{FKM}\label{sliding-sum}
For any interval $I \in \mathbb{Z}$ with $\sqrt{m}< |I|\leq m$ we have 
\begin{align*}
    \left|\sum_{n\in I}\varphi(n)\right|\leq c\sqrt{m}\log\left(\frac{4e^8|I|}{m^{1/2}}\right)
\end{align*}
    where $c=\max(||\varphi||_{\infty},||\hat{\varphi}||_{\infty})$.

\end{theorem}
Here $c$ is a bounded quantity as can be seen from \cite[equation (2.6) ]{FKM}. To be more precise with 
\begin{align*}
  \varphi(n)=\chi(f(n))e_p(g(n))
\end{align*}
with $f=f_1/f_2$ and $g=g_1/g_2$, we have 
\begin{align*}
    c\ll (\deg(f_1)+\deg(f_2)+\deg(g_1)+\deg(g_2))^2.
\end{align*}

To prove our main theorems in Section 2, we need the following two auxiliary lemmas.
\begin{lemma}\label{Aweil}
    Let $\psi$ be a non-trivial additive character on $\mathbb{F}_p$, {\it i.e.} $\psi(x)=e_p{(ax)}$ for some $a\in {\mathbb{F}_p}^{*}$ and $p$ be a prime. Then for a non-zero polynomial $G(X)\in \mathbb{Z}[X]$ of degree $d$ we have the following 
    \begin{equation*}
        \sum_{n\in I} \psi(G(n))= \mathcal{O}_d (\sqrt{p}\log p),
    \end{equation*}
  which is non-trivial provided
    $
    \sqrt{p}\log p < |I| < p.
    $

\end{lemma}
\begin{proof}
   Let $\mathcal{G} = \sum_{n \in I} \psi(G(n))$. Using completing technique,
$$
\mathcal{G}= \frac{1}{p}\sum_{b\bmod p} \lambda\left(\frac{b}{p}\right)\cdot S\left(\frac{b}{p}\right), 
$$
where
$$
S\left(\frac{b}{p}\right)= \sum_{x\bmod p}\psi(G(x))\cdot e\left(\frac{-bx}{p}\right).
$$
At this use Weil's estimate \cite{weil} to get $S\big(\frac{b}{p}\big)\ll \sqrt{p}(d-1)$ and using standard estimate $\lambda\big(\frac{b}{q}\big)\ll \min\{|I|, \frac{1}{2||\frac{b}{p}||}\} $. Hence,
\begin{align*}
    \mathcal{G} & \ll \frac{1}{p}\sum_{b\bmod p} \min\{|I|, \frac{1}{2||\frac{b}{p}||}\}.\sqrt{p}(d-1)\\ &\ll \frac{d-1}{\sqrt{p}}\{|I| + p\log p\}\\
    &\ll \sqrt{p}\log p.
\end{align*}

\end{proof}
\begin{lemma}\label{lemma-3.3}
    Let $G(X)\in \mathbb{Z}[X]$ be a polynomial of degree $d$. Then we have the following asymptotic 
\begin{equation*}
    \sum_{\substack{n\in I\\(G(n),p)=1}} 1 = |I| + \mathcal{O}_{d}(1),
\end{equation*}
provided
$ |I| < p.
    $

\end{lemma}
\begin{proof}
We have 
\begin{align*}
   \sum_{\substack{n\in I\\(F(n),p)=1}} 1&=\sum_{n \in I} \sum_{l|(F(n),p)}\mu(l) \\ &= \sum_{l|p}\mu(l)\sum_{\substack{n\in I \\ F(n)\equiv 0\bmod l}}1\\ & = |I| - \sum_{\substack{n\in I \\ F(n) \equiv 0 \bmod p}}1. 
\end{align*}
The later sum is less than equal to $d$. This concludes the proof.
\end{proof}

\section{Proof of Theorem \ref{main-1}}\label{sec-5}
We start with reducing the sum into three parts 
\begin{align*}
&\sum_{n\in I}|\chi(F(n))-\chi(F(n+1))|^{2m}\\
&=\mathcal{M}_1+\mathcal{M}_2+\mathcal{M}_3,
\end{align*}
where
\begin{align*}
\mathcal{M}_1=\sum_{\substack{n\in I\\(F(n),p)=1\\(F(n+1),p)\neq 1}}1~\text{and}~
\mathcal{M}_2=\sum_{\substack{n\in I \\(F(n),p)\neq 1\\(F(n+1),p)=1}}1,
\end{align*}
\begin{align*}
\mathcal{M}_3=\sum_{\substack{n\in I\\(F(n),p)=1\\(F(n+1),p)=1}}|\chi(F(n))-\chi(F(n+1))|^{2m}.
\end{align*}
First notice that 
\begin{align*}
\mathcal{M}_1+\mathcal{M}_2=N-\sum_{\substack{n\in I\\(F(n),p)=1\\(F(n+1),p)=1}}1-\sum_{\substack{n\in I\\(F(n),p)\neq 1\\(F(n+1),p)\neq 1}}1.
\end{align*}
Here
\begin{align*}
\sum_{\substack{n\in I\\(F(n),p)=1\\(F(n+1),p)=1}}1&=\sum_{\substack{n\in I\\(F(n)F(n+1),p)=1}}1.\\
\end{align*}
We use Lemma \ref{lemma-3.3} with $G(n)=F(n)F(n+1)$ to conclude the above sum is 
$
|I| + \mathcal{O}_{d}(1)
$
provided
$
     |I| < p.
$

 Also
\begin{align*}
\sum_{\substack{n\in I\\(F(n),p)\neq 1\\(F(n+1),p)\neq 1}}1\leq\sum_{\substack{n\in I\\F(n)\equiv F(n+1) \bmod p}}1 =\mathcal{O}_d(1).
\end{align*}

Combining all we have 
\begin{align*}
\mathcal{M}_1+\mathcal{M}_2=
\mathcal{O}_{d}(1),~ \text{provided} ~|I|<p.
\end{align*}
Now $\mathcal{M}_3$ can be written as 
\begin{align*}
\mathcal{M}_3&=\sum_{\substack{n\in I\\(F(n),p)=1\\(F(n+1),p)=1}}|\chi(F(n))-\chi(F(n+1))|^{2m}.
\end{align*}
As discussed in the Section \ref{sketch}, equation \eqref{sum-today}, we will focus on
\begin{align}\label{sumtoday1}
    \frac{1}{2^{k}}\sum_{\substack{n\in I\\(F(n),p)=1\\(F(n+1),p)=1}}\left(\chi(F(n))\overline{\chi}(F(n+1))+\chi(F(n+1))\overline{\chi}(F(n))\right)^k.
\end{align}
Expanding it we have 
\begin{align*}
    &\frac{1}{2^{k}}\sum_{\substack{n\in I\\(F(n),p)=1\\(F(n+1),p)=1}}\sum_{r=0}^{k}\binom{k}{r}\chi\left(\frac{F(n+1)}{F(n)}\right)^{k-2r}\\&
    =\frac{1}{2^{k}}\sum_{\substack{n\in I\\(F(n),p)=1\\(F(n+1),p)=1}}\sum_{\substack{r=0\\ k\equiv 2r \bmod t}}^{k}\binom{k}{r}+\frac{1}{2^{k}}\sum_{\substack{n\in I\\(F(n),p)=1\\(F(n+1),p)=1}}\sum_{\substack{r=0\\ k\not \equiv 2r \bmod t}}^{k}\binom{k}{r}\chi\left(\frac{F(n+1)}{F(n)}\right)^{k-2r}\\&
    =\frac{1}{2^{k}}\sum_{\substack{r=0\\ k\equiv 2r \bmod t}}^{k}\binom{k}{r}\sum_{\substack{n\in I\\(F(n),p)=1\\(F(n+1),p)=1}}1 + \frac{1}{2^{k}}\sum_{\substack{r=0\\ k\not \equiv 2r \bmod t}}^{k}\binom{k}{r}\sum_{\substack{n\in I\\(F(n),p)=1\\(F(n+1),p)=1}}\chi\left(\frac{F(n+1)}{F(n)}\right)^{k-2r},
\end{align*}
where Ord$(\chi)=t$. The first and the second part of the sum contribute in the main term and in the error term respectively. 
\subsection{Main term:} Let us denote $u_k \coloneqq c_k\sum\limits_{\substack{r=0\\ k\equiv 2r \pmod t}}^{k}\binom{k}{r}.$ Then our main term will be
$$
2^m\left(1-\frac{u_1}{2}+\frac{u_2}{2^2}-\frac{u_3}{2^3}+\cdot\cdot\cdot\right)\sum_{\substack{n\in I\\(F(n),p)=1\\(F(n+1),p)=1}}1.
$$
Using Lemma \ref{lemma-3.3} we get the main term 
$$
2^m\left(1-\frac{u_1}{2}+\frac{u_2}{2^2}-\frac{u_3}{2^3}+\cdot\cdot\cdot\right)\cdot |I|.
$$
Of course, it also provides some error term, which is 
$
\mathcal{O}_{d}(1).$

\subsection{Error part}
By Theorem \ref{sliding-sum}, if we take $\varphi(n)=\chi^{k-2r}\left(\frac{f(n+1)}{f(n)}\right)$, we have 
\begin{align*}
\sum_{n\in I}\chi\left(\frac{F(n+1)}{F(n)}\right)^{k-2r}\ll c\sqrt{p}\log\left(\frac{4e^8|I|}{p^{1/2}}\right)\ll \sqrt{p}\log p,
\end{align*}
which is non-trivial when $\sqrt{p}\log p<|I|<p$.

Hence combining $\mathcal{M}_1$, $\mathcal{M}_2$ and $\mathcal{M}_3$ we have the result.

\section{Proof of Theorem \ref{additivechar} }
The proof is similar to the proof of Theorem \ref{main-1}. In this case, same as \eqref{sumtoday1}, we will focus on the following
\begin{align*}
    \frac{1}{2^k}\sum\limits_{n\in I}\left(\psi(F(n))\overline{\psi(F(n+1))}+\psi(F(n+1))\overline{\psi(F(n))}\right)^k.
\end{align*}
Expanding the power we have,
\begin{align*}
    &\frac{1}{2^k}\sum\limits_{n\in I}\sum_{r=0}^{k}\binom{k}{r}\psi((k-2r)(F(n+1)-F(n)))\\&
    =\frac{1}{2^k}\sum\limits_{n\in I}\sum_{\substack{r=0\\ 2r \equiv k (p)}}^{k}\binom{k}{r}+\frac{1}{2^k}\sum\limits_{n\in I}\sum_{\substack{r=0\\ 2r\not \equiv k (p)}}^{k}\binom{k}{r}\psi\left((k-2r)(F(n+1)-F(n)\right)\\&
    =\frac{1}{2^k}\sum_{\substack{r=0\\ 2r \equiv k (p)}}^{k}\binom{k}{r}\sum\limits_{n\in I}1  +\frac{1}{2^k}\sum_{\substack{r=0\\ 2r\not \equiv k (p)}}^{k}\binom{k}{r}\sum\limits_{n\in I}\psi\left((k-2r)(F(n+1)-F(n)\right).
\end{align*}
\subsection{Main term}
Let us denote $v_k \coloneqq c_k\sum\limits_{\substack{r=0\\ 2r \equiv k (p)}}^{k}\binom{k}{r}.$ Then our main term will be
\begin{align*}
2^m\left(1-\frac{v_1}{2}+\frac{v_2}{2^2}-\frac{v_3}{2^3}+\cdot\cdot\cdot\right)\cdot |I|.
\end{align*}
\subsection{Error term}
By Lemma \ref{Aweil}, we have 
\begin{align*}
    \sum\limits_{n\in I}\psi\left((k-2r)(F(n+1)-F(n)\right)\ll\sqrt{p}\log{p},
\end{align*}
which is nontrivial when  $\sqrt{p}\log p<|I|<p$.
Combining the main term and error term we conclude the proof of Theorem \ref{additivechar}.
\section{An example}
As an example for the additive case we consider binomial coefficients as our $F(n)$. More precisely, we take $F(n)=\binom{n}{d+1}$. So
\begin{align*}
    F(n+1) - F(n) = \binom{n+1}{d+1}-\binom{n}{d+1}=\binom{n}{d}\not\equiv 0 \bmod p,
\end{align*}
for $d+1<n<p$. Hence 
$
\sum_{n \in I}|\psi(\binom{n}{d+1})-\psi(\binom{n+1}{d+1})|^{2m}
$ satisfies the formula given in Theorem \ref{additivechar}. To apply Weil bound, we treat $\frac{1}{d!}$ as inverse of $d! \bmod p$.

\section{Acknowledgment}
The first author would like to thank Department of Mathematics, Thapar Institute of Engineering and Technology and the second author would like to thank
Chennai Mathematical Institute for providing excellent working conditions. During
the preparation of this article, D.M. was supported by the National Board of Higher
Mathematics post-doctoral fellowship (No.: 0204/10/(8)/2023/R$\&$D-II/2778).


\end{document}